\documentclass[10pt]{amsart}
\usepackage{amsmath}%
\usepackage{amsfonts}%
\usepackage{amssymb}%
\usepackage{graphicx}%
\usepackage{mathtools}%
\usepackage{xcolor}%
\usepackage{nameref,hyperref,url}
\usepackage[foot]{amsaddr}

\usepackage{pgfplots}
\pgfplotsset{compat=1.15}
\usepackage{mathrsfs}
\usetikzlibrary{arrows}
\usepackage[margin=0.8in]{geometry}
\calclayout

\newtheorem{theorem}{Theorem}
\theoremstyle{plain}

\newtheorem{corollary}{Corollary}

\newtheorem{definition}{Definition}

\newtheorem{lemma}{Lemma}

\newtheorem{remark}{Remark}

\numberwithin{equation}{section}

\begin{document}

\title[Eigendecomposition of the Hessian of the Robin function in orthogonally invariant domains]{Eigendecomposition of the Hessian of the Robin function in Orthogonally invariant domains}
\author{Alejandro Ortega}
\address[A. Ortega]{Dpto. de Matem\'aticas Fundamentales, Facultad de Ciencias, UNED, 28040 Madrid, Spain}
\email{\tt alejandro.ortega@mat.uned.es}

\begin{abstract} 
In this work we analyze the eigendecomposition of the Hessian matrix of the Robin function $\mathcal{R}(x)$ for the spectral fractional Laplacian in orthogonally invariant domains. We prove that, if $\Omega$ a smooth bounded convex domain invariant under the action of an orthogonal transformation $\mathcal{O}$ then, for $\overline{t}\in\{a\in\Omega:\mathcal{O}(a)=a\}$, the gradient vector $\nabla\mathcal{R}(\overline{t})$ is an eigenvector of the Jacobian $D\mathcal{O}$ associated to the eigenvalue $1$. Moreover, if $\Omega$ is a domain invariant under the reflection about a hyperplane $\pi_{v}=\{x\in\mathbb{R}^N:x\cdot v=0\}$, there exists $\eta>0$ such that $\mathbb{H}(\overline{t})v=\eta v$ where $\mathbb{H}$ denotes the Hessian matrix of $\mathcal{R}(x)$. Consequently,  if $\Omega$ is invariant under the reflection about the hyperplanes $\pi_{v_i}$ for a linearly independent set $\{v_1,\ldots,v_N\}$, then the origin is a non degenerate critical point of $\mathcal{R}(x)$. A short proof of the Brezis--Peletier-like formulas for $\nabla \mathcal{R}$ is also provided which allows us to prove the former results for any $0<s<1$.
\end{abstract}

\maketitle


\section{Introduction and Main results}
In this work we analyze the eigendecomposition of the Hessian matrix of the Robin function for the spectral fractional Laplacian $(-\Delta)^s$, $0<s<1$, on orthogonally invariant domains $\Omega\subset\mathbb{R}^N$, $N>2s$, with $C^{2}$ boundary. Precisely, given $\mathcal{O}$ an orthogonal transformation, $\mathcal{O}(x)=Ox$ with $O$ a $N\times N$ matrix such that $OO^\top=\text{I}_{N}$, we assume that $\Omega$ is invariant under the action of $\mathcal{O}$. Let us denote by $\mathcal{F}_{\mathcal{O}}$ the set of points fixed under the action of $\mathcal{O}(x)$, $\mathcal{F}_{\mathcal{O}}=\{a\in\Omega:\mathcal{O}(a)=a\}$. Note that $\mathcal{F}_{\mathcal{O}}$ is a non-empty set as we always have $0\in\mathcal{F}_{\mathcal{O}}$. Since $O$ is an orthogonal matrix, its (maybe complex) eigenvalues $\lambda_j$ are such that $|\lambda_j|=1$ for $j=1,\ldots,N$. Let us then set $\mathcal{E}_{\mathcal{O}}=\{v\in\mathbb{R}^N\ :O\,v=v\}$ be the eigenspace associated to the eigenvalue $\lambda_1=1$.

Next, let $G_{\Omega}(x,t)$ be the Green function centered at $t\in\Omega$ of the spectral fractional Laplacian $(-\Delta)^s$. It is known that $G_{\Omega}(x,t)=G_{\mathbb{R}^N}(x,t)-H_{\Omega}(x,t)$,
where 
\begin{equation*}
G_{\mathbb{R}^N}(x,t)=\frac{c_{N,s}}{|x-t|^{N-2s}}\qquad \text{with}\qquad c_{N,s}=\frac{\Gamma\left(\frac{N-2s}{2}\right)}{2^{2s}\pi^{\frac{N}{2}}\Gamma(s)},
\end{equation*}
and $H_{\Omega}(x,t)$ is referred to as the regular part of the Green function. Actually, $H_{\Omega}(x,t)\in C^{\infty}(\Omega\times\Omega)$ for $(x,t)\in\Omega\times\Omega$ (cf. \cite[Lemma 2.4]{Choi2014}). The Robin function $\mathcal{R}(x)$ is then defined as the diagonal part of $H_{\Omega}(x,t)$,
\begin{equation*}
\mathcal{R}(x)=H_{\Omega}(x,x)\qquad \text{for }x\in\Omega.
\end{equation*}
Our main result reads as follows.
\begin{theorem}\label{ThF}
Let $\Omega\subset\mathbb{R}^N$ be a smooth bounded convex domain invariant under the action of an orthogonal transformation $\mathcal{O}(x)$. Then, 
\begin{equation}\label{pisto0}
\nabla \mathcal{R}(\overline{t})\in\mathcal{E}_{\mathcal{O}}\quad\text{for }\overline{t}\in\mathcal{F}_{\mathcal{O}}.
\end{equation}
Moreover, given $\overline{t}\in\mathcal{F}_{\mathcal{O}}$, we have
\begin{equation}\label{pisto}
\mathbb{H}(\overline{t})[\mathcal{E}_{\mathcal{O}}^\perp]\subset\mathcal{E}_{\mathcal{O}}^\perp\qquad\text{and}\qquad\mathbb{H}(\overline{t})[\mathcal{E}_{\mathcal{O}}]\subset\mathcal{E}_{\mathcal{O}}
\end{equation}
If, in addition, $\text{dim}(\mathcal{E}_{\mathcal{O}})=N-1$, so that $\mathcal{E}_{\mathcal{O}}^{\perp}=span\{v\}$, for some unitary vector $v\in\mathbb{R}^N$, then
\[\frac{\partial}{\partial v}\mathcal{R}(\overline{t})=0\quad\text{for}\ \overline{t}\in\pi_v\cap\Omega,\]
and there exists $\eta>0$ such that
\begin{equation}\label{pisto2}
\mathbb{H}(\overline{t})v=\eta v.
\end{equation}
\end{theorem}
\begin{corollary}\label{coro2}
Assume that $\Omega$ is a smooth bounded convex domain invariant 
under the reflection about the hyperplanes $\pi_{v_i}=\{x\in\mathbb{R}^N:x\cdot v_i=0\}$ for a linearly independent set $\{v_1,\ldots,v_N\}$. Then, the origin is a non degenerate critical point of the Robin function $\mathcal{R}(x)$.
\end{corollary}

The Robin function is a key element in the study of nonlinear critical elliptic problems. For instance, in the local setting $s=1$, existence and uniqueness issues for some critical problems strongly relies on the non degeneracy of critical points of the Robin function (cf.\cite{Bahri1995,Pistoia2004,Rey1990}). Moreover, positive solutions of nearly critical problems, $-\Delta u=u^p$, concentrate at exactly one point as $p\to 2^*=\frac{2N}{N-2}$ and this point is precisely a critical point of the Robin function (cf. \cite{Han1991,Rey1989} see also \cite{Rey1990}). Similar concentration results were obtained in \cite{Choi2014} for elliptic problems driven by the fractional Laplacian $s\in(0,1)$. On this point, some important properties of the Robin function have been proved in the last decades. For the local case, it is known (cf. \cite{Caffarelli1985}) that the Robin function for smooth bounded and convex domains $\Omega\subset\mathbb{R}^2$ has a unique critical point which is non degenerate. Assuming that $\Omega\subset\mathbb{R}^N$, $N\geq2$, is a smooth bounded convex domain which is symmetric with respect to the origin then (cf. \cite{Grossi2002}) the origin is a non-degenerate critical point of the Robin function. More general domains are considered in \cite{Micheletti2014}, where the authors prove the non-degeneracy of the critical points of the Robin function for small smooth deformations of domains $\Omega\subset\mathbb{R}^N$, $N\geq2$, of class $C^k$ with $k\geq4$. Regarding the fractional setting, in \cite{Ortega2023} it is proved the non-degeneracy of the Robin function at the origin by adapting the ideas of \cite{Grossi2002}. 
%
%
\section{Functional setting}\label{functionalsetting}
The definition of powers of the positive Laplace operator $-\Delta$, in a bounded domain $\Omega$ with homogeneous Dirichlet boundary data, is carried out via the spectral decomposition, using the powers of the eigenvalues of $-\Delta$ with the same boundary condition. In particular, let $(\varphi_i,\lambda_i)$ be the eigenfunctions (normalized with respect to the $L^2(\Omega)$-norm) and eigenvalues of $-\Delta$ endowed with homogeneous Dirichlet boundary data, then $(\varphi_i,\lambda_i^s)$ are the eigenfunctions and eigenvalues of $(-\Delta)^s$ with the same boundary conditions. Tehe fractional operator $(-\Delta)^s$ is well defined in the space 
\begin{equation*}
H_0^s(\Omega)=\left\{u=\sum_{j\geq 1} a_j\varphi_j\in L^2(\Omega),\ u|_{\partial\Omega}=0:\ \|u\|_{H_0^s(\Omega)}^2= \sum_{j\geq 1}
a_j^2\lambda_j^s<\infty \right\}.
\end{equation*}
As a direct consequence of the previous definition we get
\begin{equation*}
(-\Delta)^su=\sum_{j\geq 1} a_j\lambda_j^s\varphi_j\qquad\text{and}\qquad \|u\|_{H_0^s(\Omega)}=\|(-\Delta)^{\frac{s}{2}}u\|_{L^2(\Omega)}.
\end{equation*}
This definition of the fractional powers of the Laplace operator allows us to integrate by parts in the appropriate functional space. Hence, a natural definition of weak solution to the problem 
\begin{equation}\label{probgen}
        \left\{
        \begin{tabular}{rl}
        $(-\Delta)^su=f$ &in $\Omega$, \\
        $u=0$ &on $\partial\Omega$,
        \end{tabular}
        \right.
	\tag{$P_f$}
\end{equation}
is the following.
\begin{definition}
We say that $u\in H_0^s(\Omega)$ is a weak solution to problem \eqref{probgen} if
\begin{equation*}
\int_{\Omega}(-\Delta)^{s/2}u(-\Delta)^{s/2}\psi dx=\int_{\Omega}f\psi dx,\ \ \text{for all}\ \psi\in H_0^s(\Omega).
\end{equation*}
\end{definition}
Due to the nonlocal nature of the operator $(-\Delta)^s$ it is difficult or not possible to obtain explicit expressions of the action of $(-\Delta)^s$ on particular functions. We make then use of the $s$-harmonic extension introduced by Caffarelli and Silvestre (cf. \cite{Caffarelli2007}, see also \cite{Cabre2010, Braendle2013, Capella2011, Stinga2010}) which provides us with an equivalent definition of $(-\Delta)^s$ by means of an auxiliary local problem. In particular, associated with the domain $\Omega$, we consider the extension cylinder $\mathcal{C}_{\Omega}=\Omega\times(0,\infty)\subset\mathbb{R}_+^{N+1}$. We denote by $(x,y)$ points that belongs to $\mathcal{C}_{\Omega}$ and by $\partial_L\Omega=\partial\Omega\times(0,\infty)$ the lateral boundary of the extension cylinder. Note that, by its very construction, if $\nu(x)$ denotes the outward normal at $x\in\partial\Omega$, then $\nu^*(x,y)=(\nu(x),0)$ is the outward normal at $(x,y)\in\partial_L\mathcal{C}_{\Omega}$. Given $u\in H_0^s(\Omega)$, we define its $s$-extension $w=E[u](x,y)$ as the solution to

\begin{equation*}
        \left\{
        \begin{tabular}{rl}
        $-div(y^{1-2s}\nabla w(x,y))=0\mkern+26mu$  &in $\mathcal{C}_{\Omega}$, \\
        $w(x,y)=0\mkern+26mu$   &on $\partial_L\mathcal{C}_{\Omega}$, \\
        $w(x,0)=u(x)$  & in $\Omega\times\{y=0\}$.
        \end{tabular}
        \right.
\end{equation*}
The extension function belongs to the space $\mathcal{X}_0^s(\mathcal{C}_{\Omega})$ defined as the completion of $\mathcal{C}_{0}^{\infty}\left(\Omega\times[0,+\infty)\right)$ under the norm
\begin{equation*}
\|z\|_{\mathcal{X}_0^s(\mathcal{C}_{\Omega})}^2=\kappa_s\int_{\mathcal{C}_{\Omega}}y^{1-2s}|\nabla
z(x,y)|^2dxdy.
\end{equation*}
With the constant $\kappa_s=2^{2s-1}\frac{\Gamma(s)}{\Gamma(1-s)}$, (cf. \cite{Braendle2013}), the extension operator
 $E:H_0^s(\Omega)\mapsto \mathcal{X}_0^s(\mathcal{C}_{\Omega})$ is an isometry, i.e.,
\begin{equation*}
\|E[\varphi]\|_{\mathcal{X}_0^s(\mathcal{C}_{\Omega})}=\|\varphi\|_{H_0^s(\Omega)}\qquad
\text{for all}\ \varphi\in H_0^s(\Omega).
\end{equation*}
Moreover, as it is proved in \cite[Proposition 2.1]{Capella2011}, we have 
\begin{equation*}
H_0^s(\Omega)=\left\{ u=tr{\big |}_{\Omega\times\{0\}}U:\ U\in \mathcal{X}_0^s(\mathcal{C}_{\Omega})\right\},
\end{equation*}
with $tr{\big |}_{\Omega\times\{0\}} U=U(x,0)$ the trace of $U(x,y)$ on $\Omega\times\{y=0\}$. The key point of the $s$-extension function is the relation
\[
\frac{\partial w}{\partial \nu^s}:= -\kappa_s \lim_{y\to 0^+} y^{1-2s}\frac{\partial w}{\partial y}=(-\Delta)^su(x).
\]
Therefore, we can reformulate the problem \eqref{probgen} in terms of the extension problem as follows
\begin{equation}\label{extension_problem}
        \left\{
        \begin{tabular}{rl}
        $-div(y^{1-2s}\nabla w)=0$  &in $\mathcal{C}_{\Omega}$,\vspace{0.1cm} \\ 
        $w=0$   &on $\partial_L\mathcal{C}_{\Omega}$, \\
        $\displaystyle\frac{\partial w}{\partial \nu^s}=f$  & in $\Omega\times\{y=0\}$.
        \end{tabular}
        \right.
        \tag{$P_f^*$}
\end{equation}
A weak solution to \eqref{extension_problem} is a function $w\in \mathcal{X}_0^s(\mathcal{C}_{\Omega})$ such that
\begin{equation*}
\kappa_s\int_{\mathcal{C}_{\Omega}} y^{1-2s}\nabla w\cdot\nabla\varphi dxdy=\int_{\Omega} f\varphi(x,0)dx\qquad \forall \varphi\in\mathcal{X}_0^s(\mathcal{C}_{\Omega}).
\end{equation*}
Hence, if $w\in\mathcal{X}_0^s(\mathcal{C}_{\Omega})$ is a solution to \eqref{extension_problem} then
$u(x)=w(x,0)\in H_0^s(\Omega)$ and it is a solution to \eqref{probgen} and vice-versa, if $u\in H_0^s(\Omega)$ is a solution to \eqref{probgen} then $w=E[u]\in \mathcal{X}_0^s(\mathcal{C}_{\Omega})$ is a solution to \eqref{extension_problem}. Thus, both
formulations are equivalent.

\vspace{0.25cm}

 Let $G_{\Omega}(x,t)$ be the Green function centered at $t\in\Omega$ of $(-\Delta)^s$, that is,
the solution to
\begin{equation*}
        \left\{
        \begin{tabular}{rl}
        $(-\Delta)^sG_{\Omega}(x,t)=\delta_{t}$  &in $\Omega$, \\
        $G_{\Omega}(x,t)=0\mkern+5mu$  & on $\partial\Omega$.
        \end{tabular}
        \right.
\end{equation*}
The Green function $G_{\Omega}(x,t)$ is understood as the trace on $\{y=0\}$ of $G_{\mathcal{C}_{\Omega}}(x,t,y)$, the Green function of the extension problem (cf. \cite{Stinga2010}), 
\begin{equation}\label{extension_problem_green}
        \left\{
        \begin{tabular}{rl}
        $-div(y^{1-2s}\nabla_{(x,y)} G_{\mathcal{C}_{\Omega}})=0\mkern+5.5mu$  &in $\mathcal{C}_{\Omega}$,\vspace{0.1cm} \\
        $G_{\mathcal{C}_{\Omega}}=0\mkern+5.5mu$  &on $\partial_L\mathcal{C}_{\Omega}$, \\
        $\displaystyle\frac{\partial }{\partial \nu^s}G_{\mathcal{C}_{\Omega}}=\delta_{t}$  & in $\Omega\times\{y=0\}$.
        \end{tabular}
        \right.
\end{equation}
Regarding the fundamental solution of the spectral fractional Laplacian in the whole space we have (cf. \cite{Choi2014})
\[G_{\mathbb{R}^{N+1}}(x,t,y)=\frac{c_{N,s}}{|(x-t,y)|^{N-2s}},\] 
so that $\displaystyle G_{\mathbb{R}^N}(x,t)=\frac{c_{N,s}}{|x-t|^{N-2s}}$. 
Finally, the Green function for the extension problem \eqref{extension_problem_green} can be written as
\begin{equation*}
G_{\mathcal{C}_{\Omega}}(x,t,y)=G_{\mathbb{R}^{N+1}}(x,t,y)-H_{\mathcal{C}_{\Omega}}(x,t,y),
\end{equation*}
so that $H_{\Omega}(x,t)=H_{\mathcal{C}_{\Omega}}(x,t,0)$. Moreover, by \cite[Lemma 2.4]{Choi2014}, the function $t\mapsto H_{\mathcal{C}_{\Omega}}(x,t,y)$ is $C^\infty$.


\section{Proof of main results}\label{proofsresults}
In order to work in the extension cylinder, let us also set $\Upsilon(x,y)=\big(\mathcal{O}(x),y\big)$. The map $\Upsilon(x,y)$ is also an orthogonal transformation, namely, for a $(N+1)\times(N+1)$ orthogonal matrix $\Pi$ we have
\begin{equation}\label{matino}
\Upsilon(x,y)=\Pi\,(x,y)^\top.
\end{equation}
Clearly, if $\Omega$ is invariant under the action of $\mathcal{O}$ then the extension cylinder $\mathcal{C}_{\Omega}$ is invariant under the action of $\Upsilon$.
\begin{lemma}\label{lemmasymetric} Let $\Omega$ be a smooth domain invariant under the action of an orthogonal transformation $\mathcal{O}$. Then,
\begin{equation*}
G_{\mathcal{C}_{\Omega}}\big(\mathcal{O}(x),\mathcal{O}(t),y\big)=G_{\mathcal{C}_{\Omega}}(x,t,y).
\end{equation*}
In particular, for $\overline{t}\in\mathcal{F}_{\mathcal{O}}$, we have
\begin{equation*}
G_{\mathcal{C}_{\Omega}}\big(\mathcal{O}(x),\overline{t},y\big)=G_{\mathcal{C}_{\Omega}}(x,\overline{t},y).
\end{equation*}
\end{lemma}

\begin{proof}
By definition of the Green function centered at $\tau\in\Omega$ it follows that, for any  $\varphi\in C_0^{\infty}(\mathcal{C}_{\Omega})$, 
\begin{equation}\label{rev_eq1}
\kappa_s\int_{\mathcal{C}_{\Omega}}y^{1-2s}\nabla_{(x,y)} G_{\mathcal{C}_{\Omega}}(x,t,y)\cdot \nabla_{(x,y)}\varphi(x,y)dydx=\int_{\Omega}\delta_{t}\,\varphi(x,0)dx=\varphi(t,0).
\end{equation}
On the other hand, the change of variables $(z,y)=\Upsilon(x,y)$ produces
{\small \[
\begin{split}
\kappa_s\!\int_{\mathcal{C}_{\Omega}}y^{1-2s}\nabla_{(x,y)} \left(G_{\mathcal{C}_{\Omega}}\big(\mathcal{O}(x),\mathcal{O}(t),y\big)\right)\cdot\nabla_{(x,y)}\varphi(x,y)dydx=&\kappa_s\!\int_{\mathcal{C}_{\Omega}}y^{1-2s}\nabla_{(z,y)} G_{\mathcal{C}_{\Omega}}\big(z,\mathcal{O}(t),y\big)\cdot\nabla_{(z,y)}\left(\varphi\big(\mathcal{O}^{-1}(z),y\big)\right)dydz\\
=&\int_{\Omega}\delta_{\mathcal{O}(t)}\,\varphi\big(\mathcal{O}^{-1}(z),0\big)dz\\
=&\varphi(t,0),
\end{split}
\]}
since $\Upsilon$ is an orthogonal transformation so the inner product is preserved and $|det(\Pi^{-1})|=1$. Indeed,
\begin{equation*}
\nabla_{(x,y)} \left(G_{\mathcal{C}_{\Omega}}\big(\mathcal{O}(x),\mathcal{O}(t),y\big)\right)=\Pi^\top \nabla_{(z,y)} G_{\mathcal{C}_{\Omega}}\big(z,\mathcal{O}(t),y\big)\qquad\text{and}\qquad
\nabla_{(x,y)}^\top\varphi(x,y)=\Pi^\top\nabla_{(z,y)} \left(\varphi\big(\mathcal{O}^{-1}(z),y\big)\right),
\end{equation*}
where $\nabla_{(x,y)}^\top$ denotes the transpose of $\nabla_{(x,y)}$ and  $\Pi$ is the matrix \eqref{matino}. Thus, 
\begin{equation*}
\begin{split}
\nabla_{(x,y)} \left(G_{\mathcal{C}_{\Omega}}\big(\mathcal{O}(x),\mathcal{O}(t),y\big)\right)\cdot\nabla_{(x,y)}\varphi(x,y)dydx&=\nabla_{(z,y)}^\top G_{\mathcal{C}_{\Omega}}\big(z,\mathcal{O}(t),y\big)\Pi \Pi^\top \nabla_{(z,y)} \left(\varphi\big(\mathcal{O}^{-1}(z),y\big)\right)\\
&=\nabla_{(z,y)} G_{\mathcal{C}_{\Omega}}\big(z,\mathcal{O}(t),y\big)\cdot\nabla_{(z,y)}\left(\varphi\big(\mathcal{O}^{-1}(z),y\big)\right).
\end{split}
\end{equation*}
Therefore, by \eqref{rev_eq1}, we get that, for all $\varphi\in C_0^{\infty}(\mathcal{C}_{\Omega})$, 
\begin{equation}\label{bs}
\int_{\mathcal{C}_{\Omega}}y^{1-2s}\left(\nabla_{(x,y)} \left(G_{\mathcal{C}_{\Omega}}\big(\mathcal{O}(x),\mathcal{O}(t),y\big)\right)-\nabla_{(x,y)} G_{\mathcal{C}_{\Omega}}(x,t,y)\right)\cdot\nabla\varphi(x,y)dydx=0.
\end{equation}
By the invariance of $\mathcal{C}_\Omega$ under the action of $\Upsilon$, both $G_{\mathcal{C}_{\Omega}}\big(\mathcal{O}(x),\mathcal{O}(t),y\big)$ and $G_{\mathcal{C}_{\Omega}}(x,t,y)$ take the same value $0$ on $\partial_L\mathcal{C}_{\Omega}$. This, together with \eqref{bs}, implies $G_{\mathcal{C}_{\Omega}}\big(\mathcal{O}(x),\mathcal{O}(t),y\big)=G_{\mathcal{C}_{\Omega}}(x,t,y)$ a.e. in $\mathcal{C}_{\Omega}$.
By continuity of the Green function for the extension problem (except at the singular point $(t,0)\in\Omega$), the identity \eqref{bs} holds everywhere in $\mathcal{C}_{\Omega}$. 
\end{proof}
\begin{remark}\label{remark_uno}
Lemma \ref{lemmasymetric} above can be equivalently restated as $G_{\mathcal{C}_{\Omega}}\big(\mathcal{O}^{-1}(x),t,y\big)=G_{\mathcal{C}_{\Omega}}(x,\mathcal{O}(t),y)$.
\end{remark}
Being $\mathcal{O}$ an orthogonal map, $|\mathcal{O}(x)-\mathcal{O}(t)|\!=\!|x-t|$ so $G_{\mathbb{R}^{N+1}}\big(\mathcal{O}(x),\mathcal{O}(t),y\big)=G_{\mathbb{R}^{N+1}}(x,t,y)$. Thus, by Lemma \ref{lemmasymetric},
\begin{equation}\label{prepa}
H_{\mathcal{C}_{\Omega}}(\mathcal{O}(x),\mathcal{O}(t),y)=H_{\mathcal{C}_{\Omega}}(x,t,y)
\end{equation} 
and, as a consequence, $H_{\Omega}(\mathcal{O}(x),\mathcal{O}(t))=H_{\Omega}(x,t)$ so that
\begin{equation}\label{padre}
\mathcal{R}(\mathcal{O}(x))=\mathcal{R}(x).
\end{equation}
The former equality entails the proof of \eqref{pisto0} by noticing that the chain rule produces 
\begin{equation}\label{pd2}
O^\top \nabla \mathcal{R}(\mathcal{O}(x))=\nabla \mathcal{R}(x).
\end{equation} 
Therefore, $O\nabla \mathcal{R}(x)= \nabla \mathcal{R}(\mathcal{O}(x))$ so taking $\overline{t}\in\mathcal{F}_{\mathcal{O}}$ we get $O \nabla \mathcal{R}(\overline{t})=\nabla \mathcal{R}(\overline{t})$. Thus, 
\begin{equation}\label{easy}
\nabla \mathcal{R}(\overline{t})\in\mathcal{E}_{\mathcal{O}}\quad\text{for}\ \overline{t}\in\mathcal{F}_{\mathcal{O}}.
\end{equation}
Then, the chain rule jointly and \eqref{pd2} imply that the Hessian of the Robin function satisfies $O^\top \mathbb{H}(\mathcal{O}(x)) O=\mathbb{H}(x)$, so 
\begin{equation}\label{posteasy}
O^\top \mathbb{H}(\overline{t}) O=\mathbb{H}(\overline{t})\quad \text{for}\ \overline{t}\in\mathcal{F}_{\mathcal{O}}.
\end{equation}
Since by \eqref{posteasy} above the Hessian conmutes with the matrix $O$, say $\mathbb{H}(\overline{t}) O=O \mathbb{H}(\overline{t})$ for $\overline{t}\in\mathcal{F}_{\mathcal{O}}$, we find
 \begin{equation}\label{inclusiones}
 \mathbb{H}(\overline{t})[\mathcal{E}_{\mathcal{O}}]\subset \mathcal{E}_{\mathcal{O}}\qquad\text{and}\qquad \mathbb{H}(\overline{t})[\mathcal{E}_{\mathcal{O}}^\perp]\subset \mathcal{E}_{\mathcal{O}}^\perp,
 \end{equation}
 and \eqref{pisto2} follows. Actually, the orthogonal decomposition $\mathbb{R}^N=\mathcal{E}_{\mathcal{O}}\oplus\mathcal{E}_{\mathcal{O}}^\perp$ implies the block structure
 \[\mathbb{H}(\overline{t})=\begin{pmatrix}
 H_{1}&0\\
 0& H_{\perp}
 \end{pmatrix}.\]
The next examples illustrate the above results. If $\Omega$ is invariant under the action of an orthogonal transformation $\mathcal{O}$ that has no nontrivial fixed points, say, $\mathcal{F}_{\mathcal{O}}=\{0\}$, we have $\mathcal{E}_{\mathcal{O}}=\{0\}$. Then, \eqref{easy} implies $\nabla\mathcal{R}(0)=0$ so the origin is a critical point of the Robin function. If, for example, $\Omega$ is symmetric respect to $x_1$, then it is invariant under the action of $\mathcal{O}(x_1,x_2,\ldots, x_N)=(-x_1,x_2,\ldots,x_N)$, so $\mathcal{E}_{\mathcal{O}}=\{0\}\times\mathbb{R}^{N-1}$ and hence $\frac{\partial }{\partial x_1} \mathcal{R}(\overline{t})=0$ for $\overline{t}\in\Omega\cap\{x_1=0\}$. Furthermore, in this case \eqref{posteasy} implies 
\[\mathbb{H}(\overline{t})=\begin{pmatrix}
h_{11}&0\\
0& H_{\perp}
\end{pmatrix}\qquad\text{for}\ \overline{t}\in\Omega\cap\{x_1=0\}\ \text{and some}\ h_{11}\in\mathbb{R}.\]
Indeed, if $dim(\mathcal{E}_{\mathcal{O}})=N-1$ so that $\mathcal{E}_{\mathcal{O}}^\perp=span\{v\}$, by \eqref{inclusiones} it follows that 
\[\mathbb{H}(\overline{t}) v=\eta v\qquad\text{for some}\ \eta\in\mathbb{R}.\]
The main aim in what follows is then to prove that the number $\eta$ has actually a sign, say $\eta\neq0$, so the second derivative of the Robin function in the normal direction to the invariant subspace is not degenerated. The proof of this fact will follow by means of a PDE based argument which, as a byproduct, will also produce the results derived from \eqref{padre}. 

Given $t\in\Omega$, let us consider the vector function $\mathcal{U}(x,t,y)$ solution to 
\begin{equation}\label{g2}
        \left\{
        \begin{tabular}{rl}
        $-div(y^{1-2s}\nabla_{(x,y)} \mathcal{U} )=0\mkern+104mu$  &in $\mathcal{C}_{\Omega}$, \\
        $\mathcal{U}(x,t,y)=\nabla_{x}G_{\mathcal{C}_{\Omega}}(x,t,y)$  &on $\partial_L\mathcal{C}_{\Omega}$, \\
        $\frac{\partial }{\partial \nu^s}\mathcal{U}(x,t,0)=0\mkern+104mu$  & in $\Omega\times\{y=0\}$,
        \end{tabular}
        \right.			
\end{equation}
namely, $\mathcal{U}(x,t,y)=(U_1(x,t,y),\ldots,U_N(x,t,y))^\top$ with $U_j(x,t,y)$ being the solution to the problem 
\begin{equation}\label{pu1}
        \left\{
        \begin{tabular}{rl}
        $-div(y^{1-2s}\nabla_{(x,y)} U_j)=0\mkern+133mu$  &in $\mathcal{C}_{\Omega}$, \\
        $U_j(x,t,y)=e_j\cdot\nabla_x G_{\mathcal{C}_{\Omega}}(x,t,y)$  &on $\partial_L\mathcal{C}_{\Omega}$, \\
        $\frac{\partial }{\partial \nu^s}U_j(x,t,0)=0\mkern+133mu$  & in $\Omega\times\{y=0\}$.
        \end{tabular}
        \right.			
\end{equation}
\begin{lemma}\label{lemmaodd}
For all $x,t\in\Omega$ and $y\geq0$ it holds that
\[
O\,\mathcal{U}(x,t,y)=\mathcal{U}(\mathcal{O}(x),\mathcal{O}(t),y).
\]
\end{lemma}
\begin{proof}
By the very definition of the Green function and the regularity properties of $H_{\mathcal{C}_\Omega}(x,t,y)$, it follows that the function $\mathcal{U}(x,t,y)=-(\nabla_x+\nabla_t)H_{\mathcal{C}_\Omega}(x,t,y)$ is the solution to the problem \eqref{g2}. Thus, by \eqref{prepa}, we have
\[
\begin{split}
\mathcal{U}(x,t,y)&=-(\nabla_x+\nabla_t)H_{\mathcal{C}_\Omega}(x,t,y)=-(\nabla_x+\nabla_t)H_{\mathcal{C}_{\Omega}}(\mathcal{O}(x),\mathcal{O}(t),y)=-O^\top(\nabla_z+\nabla_\tau)H_{\mathcal{C}_{\Omega}}(z,\tau,y)=O^\top\mathcal{U}(z,\tau,y)\\
&=O^\top \mathcal{U}(\mathcal{O}(x),\mathcal{O}(t),y).
\end{split}\]
Next we provide another proof of this result for the trace of $\mathcal{U}(x,t,y)$ on $\{y=0\}$ by means of the Poisson-like representation formula for the solution to \eqref{pu1}. Combining this representation formula together with the explicit expression $\mathcal{U}(x,t,y)=-(\nabla_x+\nabla_t)H_{\mathcal{C}_\Omega}(x,t,y)$ we obtain a new and very short proof of the Brezis--Peletier-like formulas for $\nabla \mathcal{R}(t)$ (cf. \cite{Brezis1989}) stated in Lemma \ref{gradrobin} below. In particular, multiplying the first equation of \eqref{pu1} by $G_{\mathcal{C}_{\Omega}}(x,\tau,y)$ and integrating by parts twice, we get the representation formula
\begin{equation}\label{representation}
U_j(\tau,t,0)=-\kappa_s\int_{\partial_L\mathcal{C}_{\Omega}} y^{1-2s}\frac{\partial}{\partial x_j}G_{\mathcal{C}_{\Omega}}(x,t,y)\frac{\partial }{\partial\nu_{(x,y)}^*}G_{\mathcal{C}_{\Omega}}(x,\tau,y)d\sigma_{(x,y)}.
\end{equation}
Equivalently, the above representation formula can be stated in terms of the vector function $\mathcal{U}$ as  
\begin{equation}\label{representation2}
\mathcal{U}(\tau,t,0)=-\kappa_s\int_{\partial_L\mathcal{C}_{\Omega}} y^{1-2s}\nabla_x G_{\mathcal{C}_{\Omega}}(x,t,y)\frac{\partial }{\partial\nu_{(x,y)}^*}G_{\mathcal{C}_{\Omega}}(x,\tau,y)d\sigma_{(x,y)}.
\end{equation}
Then, because of Lemma \ref{lemmasymetric},
\[
\mathcal{U}(\tau,t,0)=-\kappa_s\int_{\partial_L\mathcal{C}_{\Omega}} y^{1-2s}\nabla_x\left(G_{\mathcal{C}_{\Omega}}\big(\mathcal{O}(x),\mathcal{O}(t),y\big)\right)\frac{\partial }{\partial\nu_{(x,y)}^*}G_{\mathcal{C}_{\Omega}}(x,\tau,y)d\sigma_{(x,y)}.
\]
Using that $G_{\mathcal{C}_{\Omega}}(x,\tau,y)=0$ on $\partial_L\mathcal{C}_{\Omega}$ jointly with Lemma \ref{lemmasymetric} (see Remark \ref{remark_uno}), we get
\[
\begin{split}
\frac{\partial }{\partial\nu_{(x,y)}^*}G_{\mathcal{C}_{\Omega}}(x,\tau,y)&=-|\nabla_{(x,y)}G_{\mathcal{C}_{\Omega}}(x,\tau,y)|=-|\Pi^\top\nabla_{(z,y)} \left(G_{\mathcal{C}_{\Omega}}(\mathcal{O}^{-1}(z),\tau,y)\right)|
=-|\nabla_{(z,y)} G_{\mathcal{C}_{\Omega}}(z,\mathcal{O}(\tau),y)|\\[-3pt]
&=\frac{\partial }{\partial\nu_{(z,y)}^*}G_{\mathcal{C}_{\Omega}}(z,\mathcal{O}(\tau),y).
\end{split}
\]
Thus, the change of variable $z=\mathcal{O}(x)$ produces
\[
\mathcal{U}(\tau,t,0)=-\kappa_s\int_{\partial_L\mathcal{C}_{\Omega}} y^{1-2s}O^\top\nabla_{z} G_{\mathcal{C}_{\Omega}}\big(z,\mathcal{O}(t),y\big)\frac{\partial }{\partial\nu_{(z,y)}^*}G_{\mathcal{C}_{\Omega}}(z,\mathcal{O}(\tau),y)d\sigma_{(z,y)}
=O^\top\mathcal{U}(\mathcal{O}(\tau),\mathcal{O}(t),0).\]
\end{proof}
Next, fixed $\overline{t}\in\mathcal{F}_{\mathcal{O}}$ and a vector $v\in\mathbb{R}^N$, let $\mathcal{V}(x)=V(x,0)$ with $V(x,y)$ being the solution to
\begin{equation}\label{guapa}
        \left\{
        \begin{tabular}{rl}
        $-div(y^{1-2s}\nabla V)=0\mkern+104mu$  &in $\mathcal{C}_{\Omega}$, \\
        $V(x,y)=\dfrac{\partial}{\partial v}G_{\mathcal{C}_{\Omega}}(x,\overline{t},y)$  &on $\partial_L\mathcal{C}_{\Omega}$, \\
        $\frac{\partial }{\partial \nu^s}V(x,0)=0\mkern+104mu$  & in $\Omega\times\{y=0\}$.
        \end{tabular}
        \right.			
\end{equation} 
\begin{lemma}\label{lemmaguapo}
If $v\in\mathcal{E}_{\mathcal{O}}^\perp$ then $\mathcal{V}(t)=0$ for ${t}\in\mathcal{F}_{\mathcal{O}}$. As a consequence, 
\begin{equation}\label{dht}
\nabla \mathcal{V}(x)\big|_{x={t}}\in\mathcal{E}_{\mathcal{O}}^\perp,\quad\text{for}\ {t}\in\mathcal{F}_{\mathcal{O}}.
\end{equation}
Moreover, if $dim(\mathcal{E}_{\mathcal{O}})=N-1$, then
\vspace{-0.2cm}
\begin{equation}\label{tabaco}
\frac{\partial}{\partial v}\mathcal{V}(x)\Big|_{x={t}}<0,\quad\text{for}\ t\in\mathcal{F}_{\mathcal{O}}.
\end{equation}
\end{lemma}
\begin{proof}
As, by linearity, $\mathcal{V}(x)=\mathcal{U}(x,\overline{t},0)\cdot v$ and, 
by Lemma \ref{lemmaodd}, we have $\mathcal{U}(x,\overline{t},0)\in\mathcal{E}_{\mathcal{O}}$ for $x\in\mathcal{F}_{\mathcal{O}}$, given $v\in\mathcal{E}_{\mathcal{O}}^\perp$ we get
\begin{equation}\label{celo}
\mathcal{V}(t)=0,\quad\text{for}\ {t}\in\mathcal{F}_{\mathcal{O}}.
\end{equation}
Being $\mathcal{V}(x)$ constant along $\mathcal{F}_{\mathcal{O}}=span\{\mathcal{E}_{\mathcal{O}}\}\cap\Omega$, its gradient is orthogonal to $\mathcal{E}_{\mathcal{O}}$ and \eqref{dht} follows. Next, assume that $dim(\mathcal{E}_{\mathcal{O}})=N-1$, so that $\mathcal{E}_{\mathcal{O}}^\perp=span\{v\}$. Hence, $\mathcal{F}_{\mathcal{O}}=\pi_v\cap\Omega$, where $\pi_v=\{x\in\mathbb{R}^N:x\cdot v=0\}$. Since $\Omega$ is a convex domain, the hyperplane $\pi_v$ splits $\Omega$ in two sub-domains $\Omega_1=\{x\in\Omega: x\cdot v<0\}$ and $\Omega_2=\{x\in\Omega: x\cdot v>0\}$, so that 
\begin{equation}\label{pope}
v\cdot\nu(x)\leq0\quad\text{for all}\ x\in\partial\Omega_1\qquad\text{and}\qquad v\cdot\nu(x)\geq0\quad\text{for all}\ x\in\partial\Omega_2.
\end{equation}
Moreover, since $G_{\mathcal{C}_{\Omega}}(x,\overline{t},y)>0$ for $(x,y)\in\mathcal{C}_{\Omega}$ and $G_{\mathcal{C}_{\Omega}}(x,\overline{t},y)=0$ for $(x,y)\in\partial_L\mathcal{C}_{\Omega}$, we also have 
\[
\nu^*(x,y)=-\frac{\nabla_{(x,y)} G_{\mathcal{C}_{\Omega}}(x,\overline{t},y)}{|\nabla_{(x,y)} G_{\mathcal{C}_{\Omega}}(x,\overline{t},y)|}\quad\text{at } (x,y)\in\partial_L\mathcal{C}_{\Omega}.
\]
This, jointly with \eqref{pope}, produces
\begin{equation}\label{yija}
\dfrac{\partial}{\partial v}G_{\mathcal{C}_{\Omega}}(x,\overline{t},y)\geq0\ \text{in } \partial_L\mathcal{C}_{\Omega}\cap\{x\cdot v<0\}\quad\text{and}\quad\dfrac{\partial}{\partial v}G_{\mathcal{C}_{\Omega}}(x,\overline{t},y)\leq0\ \text{in } \partial_L\mathcal{C}_{\Omega}\cap\{x\cdot v>0\}.
\end{equation}
Then, by \eqref{celo} and \eqref{yija}, the function $V$ satisfies 
\begin{equation*}
        \left\{
        \begin{tabular}{rl}
        $-div(y^{1-2s}\nabla V)=0$  &in $\mathcal{C}_{\Omega_1}$, \\
        $V(x,y)\geq0$  &on $\partial_L\mathcal{C}_{\Omega_1}$, \\
        $\frac{\partial }{\partial \nu^s}V(x,0)=0$  & in $\Omega_1\times\{y=0\}$.
        \end{tabular}
        \right.			
\end{equation*}
Hence, by the Maximum Principle (cf. \cite[Lemma 2.6]{Capella2011}), we have $V(x,0)>0$ in $\Omega_1$. Applying the Hopf Lemma to $V(x,y)$ in $\mathcal{C}_{\Omega_1}$, (see \cite[Lemma 2.7]{Capella2011}) we conclude 
\begin{equation*}
\frac{\partial}{\partial v}V(x,0)\Big|_{x={t}}<0,\quad\text{for}\ {t}\in\mathcal{F}_{\mathcal{O}}.
\end{equation*}
\end{proof}
The proof of Theorem \ref{ThF} also relies in the next Lemma \ref{gradrobin} which extends to the fractional setting \cite[Theorem 4.4]{Brezis1989}.
\begin{lemma}\cite[Lemma 3]{Ortega2023}\label{gradrobin} Let $\frac12<s<1$ and $\Omega$ be a smooth bounded domain with $C^2$ boundary. Then, for any $p\in\Omega$,
\begin{equation}\label{grad0}
\nabla\mathcal{R}(t)=\kappa_s\int_{\partial_L\mathcal{C}_{\Omega}}y^{1-2s}\left(\frac{\partial G_{\mathcal{C}_{\Omega}}}{\partial\nu_{(x,y)}^*}(x,t,y)\right)^2\nu(x) d\sigma_{(x,y)},
\end{equation}
that is,
\vspace{-0.2cm}
\begin{equation}\label{grad}
\frac{\partial}{\partial t_i}\mathcal{R}(t)=\kappa_s\int_{\partial_L\mathcal{C}_{\Omega}}y^{1-2s}\left(\frac{\partial G_{\mathcal{C}_{\Omega}}}{\partial\nu_{(x,y)}^*}(x,t,y)\right)^2\nu_i(x) d\sigma_{(x,y)},
\end{equation}
and
\vspace{-0.2cm}
\begin{equation}\label{grad2}
\frac{\partial^2}{\partial t_i\partial t_j}\mathcal{R}_{\Omega}^s(t)=2\kappa_s\int_{\partial_L\mathcal{C}_{\Omega}}y^{1-2s}\frac{\partial G_{\mathcal{C}_{\Omega}}}{\partial t_i}(t,x,y)\frac{\partial}{\partial t_j}\left(\frac{\partial G_{\mathcal{C}_{\Omega}}}{\partial\nu_{(x,y)}^*}(x,t,y)\right) d\sigma_{(x,y)}.
\end{equation}
\end{lemma}
\begin{remark}
{\rm The hypothesis $\frac{1}{2}<s<1$ assumed in Lemma \ref{gradrobin} underlies in the regularity results (cf. \cite[Theorem 1.5]{Caffarelli2016}) used in the proof of the identity \eqref{grad} in \cite{Ortega2023}. However, the explicit expression $\mathcal{U}(x,t,0)=-(\nabla_x+\nabla_t)H_{\Omega}(x,t)$ jointly with \eqref{representation2} produces a short proof of \eqref{grad0} that avoids such restriction. Actually, contrary to the argument in \cite{Ortega2023} in which a convergence of approximated functions needs to be justified by means of elliptic regularity, this approach directly produces
\[\begin{split}
\nabla \mathcal{R}(p)&=(\nabla_x+\nabla_t)H_{\Omega}(x,t)\big|_{(x,t)=(p,p)}=\mathcal{U}(p,p,0)=\kappa_s\int_{\partial_L\mathcal{C}_{\Omega}} y^{1-2s}\nabla_x G_{\mathcal{C}_{\Omega}}(x,p,y)\frac{\partial }{\partial\nu_{(x,y)}^*}G_{\mathcal{C}_{\Omega}}(x,p,y)d\sigma_{(x,y)}\\
&=\kappa_s\int_{\partial_L\mathcal{C}_{\Omega}}y^{1-2s}\left(\frac{\partial G_{\mathcal{C}_{\Omega}}}{\partial\nu_{(x,y)}^*}(x,p,y)\right)^2\nu(x) d\sigma_{(x,y)}.
\end{split}\]
}
\end{remark}
\begin{proof}[Proof of Theorem \ref{ThF}]
Because of \eqref{representation} and Lemma \ref{gradrobin},
\[
\begin{split}
U_j(\overline{t},\overline{t},0)&=-\kappa_s\int_{\partial_L\mathcal{C}_{\Omega}} y^{1-2s}\frac{\partial}{\partial x_j}G_{\mathcal{C}_{\Omega}}(x,\overline{t},y)\frac{\partial }{\partial\nu_{(x,y)}^*}G_{\mathcal{C}_{\Omega}}(x,\overline{t},y)d\sigma_{(x,y)}\\
&=-\kappa_s\int_{\partial_L\mathcal{C}_{\Omega}}y^{1-2s}\left(\frac{\partial G_{\mathcal{C}_{\Omega}}}{\partial\nu_{(x,y)}^*}(x,\overline{t},y)\right)^2\nu_i(x) d\sigma_{(x,y)}\\
&=-\frac{\partial}{\partial t_j}\mathcal{R}(\overline{t}),
\end{split}
\]
that is, $\mathcal{U}(\overline{t},\overline{t},0)=-\nabla\mathcal{R}(\overline{t})$ for $\overline{t}\in\mathcal{F}_{\mathcal{O}}$. Due to Lemma \ref{lemmaodd}, we conclude
\[
\nabla \mathcal{R}(\overline{t})\in\mathcal{E}_{\mathcal{O}}\quad\text{for}\ \overline{t}\in\mathcal{F}_{\mathcal{O}}.
\]
Next, since  $G_{\mathcal{C}_{\Omega}}(x,t,y)=G_{\mathcal{C}_{\Omega}}(t,x,y)$, we have 
\begin{equation}\label{partiali}
\frac{\partial}{\partial x_j}G_{\mathcal{C}_{\Omega}}(x,t,y)=\frac{\partial}{\partial t_j}G_{\mathcal{C}_{\Omega}}(t,x,y).
\end{equation}
Then, differentiating \eqref{representation} with respect to $t_i$ and using \eqref{partiali} together with \eqref{grad2}, we get 
\begin{equation}\label{llp}
\begin{split}
\frac{\partial}{\partial t_i}U_j(t,\overline{t},0)\big|_{t=\overline{t}}&=-\kappa_s\int_{\partial_L\mathcal{C}_{\Omega}} y^{1-2s}\frac{\partial}{\partial x_j}G_{\mathcal{C}_{\Omega}}(x,\overline{t},y)\frac{\partial}{\partial t_i}\left(\frac{\partial }{\partial\nu_{(x,y)}^*}G_{\mathcal{C}_{\Omega}}(x,\overline{t},y)\right)d\sigma_{(x,y)}\\
&=-\kappa_s\int_{\partial_L\mathcal{C}_{\Omega}} y^{1-2s}\frac{\partial}{\partial t_j}G_{\mathcal{C}_{\Omega}}(\overline{t},x,y)\frac{\partial}{\partial t_i}\left(\frac{\partial }{\partial\nu_{(x,y)}^*}G_{\mathcal{C}_{\Omega}}(x,\overline{t},y)\right)d\sigma_{(x,y)}\\
&=-\frac{1}{2}\frac{\partial^2}{\partial t_i\partial t_j}\mathcal{R}(\overline{t}).
\end{split}
\end{equation}
Let $\mathcal{V}(x)=V(x,0)$ with $V(x,y)$ the solution to \eqref{guapa} with $v\in\mathcal{E}_{\mathcal{O}}^\perp$. Then, the equation \eqref{llp} produces,
\begin{equation*}
\frac{\partial}{\partial t_i}\mathcal{V}(t)\big|_{t=\overline{t}}=\frac{\partial}{\partial t_i}\left(\mathcal{U}(t,\overline{t},0)\cdot v\right)\big|_{t=\overline{t}}=\sum_{j=1}^{N}\left(\frac{\partial}{\partial t_i}U_j(t,\overline{t},0)\big|_{t=\overline{t}}\right)v_j=-\frac{1}{2}\sum_{j=1}^{N}\frac{\partial^2}{\partial t_i\partial t_j}\mathcal{R}(\overline{t})v_j.
\end{equation*}
Denoting by $\mathbb{H}(t)$ the Hessian Matrix of the Robin function, the above reads
\begin{equation}\label{loki2}
\nabla \mathcal{V}(\overline{t})=-\frac{1}{2}\mathbb{H}(\overline{t}) v,
\end{equation}
Then, the first inclusion in \eqref{pisto} follows by \eqref{dht}. Moreover, since $\mathbb{H}$ is a symmetric matrix, given $v\in\mathcal{E}_{\mathcal{O}}$, we have $\mathbb{H}(\overline{t}) v\cdot w=v\cdot\mathbb{H}(\overline{t}) w=0$,
because $\mathbb{H}(\overline{t}) w\in\mathcal{E}_{\mathcal{O}}^\perp$ for any $w\in\mathcal{E}_{\mathcal{O}}^\perp$. Hence, $\mathbb{H}(\overline{t})v\in \mathcal{E}_{\mathcal{O}}$ for $v\in \mathcal{E}_{\mathcal{O}}$. Finally, assume that $dim(\mathcal{E}_{\mathcal{O}})=N-1$ so $\mathcal{E}_{\mathcal{O}}^\perp=span\{v\}$ for some unitary vector $v\in\mathbb{R}^N$. Then, since $\nabla \mathcal{R}(\overline{t})\in\mathcal{E}_{\mathcal{O}}$ for $\overline{t}\in\mathcal{F}_{\mathcal{O}}=\pi_v\cap\Omega$,
\begin{equation}\label{sao1}
\frac{\partial}{\partial v}\mathcal{R}(\overline{t})=\nabla \mathcal{R}(\overline{t})\cdot v=0,\qquad\text{for}\ \overline{t}\in \pi_v\cap\Omega.
\end{equation}
Moreover, by \eqref{tabaco},
\begin{equation}\label{p2p}
0>\frac{\partial}{\partial v}\mathcal{V}(\overline{t})=\nabla\mathcal{V}(\overline{t})\cdot v=-\frac{1}{2}(\mathbb{H}(\overline{t}) v)\cdot v=-\frac12 v^\top\mathbb{H}(\overline{t}) v,
\end{equation}
and \eqref{pisto2} follows. Indeed, fixed $\overline{t}\in\mathcal{F}_{\mathcal{O}}=\pi_v\cap\Omega$, by \eqref{dht} we get $\nabla\mathcal{V}(\overline{t})=\mu v\quad\text{for some}\ \mu\in\mathbb{R}$. Then, \eqref{loki2} gives us
\begin{equation}\label{sao2}
\mathbb{H}(\overline{t}) v=-2\mu v.
\end{equation}
Since, by \eqref{p2p}, we have $\mu<0$, then \eqref{pisto2} follows by taking $\eta=-2\mu>0$.
\end{proof}
\begin{proof}[Proof of Corollary \ref{coro2}]
Let $v_1,\ldots,v_k\in\mathbb{R}^N$, $1\leq k\leq N$, be linearly independent unitary vectors and take $\mathcal{F}_{\cap}$ the set of points invariant under the action of each $\Pi_{v_j}$ for $j=1,\ldots,k$, namely, $\mathcal{F}_{\cap}=\bigcap_{j=1}^{k}(\pi_{v_j}\cap\Omega)$.
Let us also take $\mathcal{E}_{\cap}=\{w\in\mathbb{R}^N:\ \Pi_{v_j}w=w\ \text{for all}\ j=1,\ldots, k\}$. Clearly, $\mathcal{E}_{\cap}=(span\{v_1,\ldots,v_k\})^\perp$ so $dim(\mathcal{E}_{\cap})= N-k$. By \eqref{sao1} and \eqref{sao2}, given $\overline{t}\in\mathcal{F}_{\cap}$, we have
\[
\frac{\partial}{\partial v_j}\mathcal{R}(\overline{t})=0\qquad\text{and}\qquad \mathbb{H}(\overline{t}) v_j=\eta_jv_j\quad\text{with }\eta_j>0\qquad\text{for}\ j=1,\ldots,k.\]
Thus, $\mathbb{H}(\overline{t})$ has $1\leq k\leq N$ positive eigenvalues.
If $k=N$ then $\mathcal{F}_{\cap}=\{0\}$ and $\mathcal{E}_{\cap}=\{0\}$. By Lemma \ref{lemmaodd}, we have $\mathcal{U}(0,0,0)\in\mathcal{E}_{\cap}$ so that $\mathcal{U}(0,0,0)=\{0\}$. Then, the origin is a critical point of the Robin function and the Hessian matrix $\mathbb{H}(0)$ is a positive definite matrix with eigenvectors $v_1,\ldots,v_N$ associated to the eigenvalues $\lambda_1,\ldots,\lambda_N$ respectively. In particular, the origin is a non degenerated critical point of $\mathcal{R}(x)$.
\end{proof}

\section*{Acknowledgment}
\noindent A. Ortega is partially funded by the Vicerrectorado de Investigación, Transferencia y Divulgación Científica of the Universidad Nacional de Educación a Distancia under the Plan de Promoción de Investigación through the research project \textit{Proyectos Investigación Tipo A. EDP no locales con dato de frontera mixto Dirichlet-Neumann}, Ref. 2026/00131/001.



\end{document}